\documentclass[12pt]{article}
\usepackage{amsmath,amssymb,amsthm}

\newtheorem{definition}{Definition}
\newtheorem{theorem}[definition]{Theorem}

\newtheorem{corollary}[definition]{Corollary}

\newtheorem*{definition*}{Definition}
\newtheorem*{theorem*}{Theorem}
\newtheorem*{proposition*}{Proposition}
\newtheorem*{example*}{Example}
\newtheorem*{exercise*}{Exercise}
\newtheorem*{corollary*}{Corollary}
\newtheorem*{remark*}{Remark}

\usepackage{geometry}
\usepackage[shortlabels, inline]{enumitem}
\usepackage{bm}
\usepackage{graphicx}
\usepackage{float}

\allowdisplaybreaks[2]
\usepackage{color}
\usepackage{mathtools}
\mathtoolsset{showonlyrefs=true}
\usepackage{url}

\begin{document}


\begin{center}
  ~\vspace{20pt}
  
  \Large 
  Analysis of an Inhomogeneous Random Walk with Spatial Decay of Transition Probabilities and Parameter Renewal per Excursion

  \vspace{20pt}
  
  \large
  Naohiro Yoshida

  \normalsize
  Department of Economics,
  Keiai University

  1-5-21,
  Anagawa, Inage,
  263-8588,
  Chiba,
  Japan

  E-mail:
  \url{n-yoshida@u-keiai.ac.jp}

\end{center}

  \vspace{20pt}

\noindent
  MSC2020:
  60G50 ; 60J10 ; 60K37
  
\noindent
  Keywords:
  inhomogeneous random walk ; excursion theory ; random environment ; scale function ; hitting time ; occupation time

\begin{abstract}
  In this paper, we propose and analyze a novel one-dimensional inhomogeneous random walk model that combines spatial decay of transition probabilities with a temporal renewal structure for each excursion. In this model, the probability of moving to the right from each state creats a spatial inhomogeneity that causes a stronger pull-back toward the origin as the process moves farther away. Furthermore, it features a random environment aspect where the parameter of each transition probability is independently resampled from a uniform distribution at the beginning of each excursion. We rigorously derive the hitting probability to an upper boundary using a scale function. Furthermore, by solving linear difference equations, we provide the probability generating function of the first hitting time, the expected occupation time for each state during an excursion (discrete Green's function), and the distribution and expectation of the maximum penetration depth.
\end{abstract}

\newpage
\section{Introduction}

The analysis of "excursions" in random walks, that is, the trajectory from leaving the origin until returning to it, is a classical and central theme in probability theory. In the framework of the simple symmetric random walk (SSRW), the asymptotic behavior of the length and height distributions of excursions, or their order statistics, has been investigated in detail by, e.g., \cite{CsakiHu2003, CsakiErdosRevesz1985, CsakiMohanty1981,Kaigh1978}, including their connection to the local time of Brownian motion.

The applications of excursion theory are diverse, playing important roles in tests of randomness in \cite{BaronRukhin1998} and the pricing of financial options in \cite{FujitaKawanishiYor2014, FujitaYor2007}. In recent years, \cite{FujitaYoshida2023, FujitaYoshida2025} have reported new developments, such as proofs of the arcsine law for simple random walks using excursions associated with specific local times and the method of marked excursions. Furthermore, \cite{FujitaYoshida2023_JSIAML} proposed a framework that applies discrete-time excursion theory to evaluate ``excursion risk" in financial investments, reaffirming its importance from a practical perspective.

As a recent research trend, the analysis of models that add specific modulations or inhomogeneities to standard random walks is actively conducted. As preceding studies closely related to this research, we first mention the work by \cite{pilipenko2017limit} on a model where transition probabilities are modified every time it returns to the origin. They discuss the impact of such modifications per visit on limit theorems, supporting the mathematical validity of the renewal structure of ``parameter resampling per excursion" in our model. Additionally, Engländer and Volkov \cite{englander2019impatient} proposed the impatient random walk, where sojourn times and movement rules depend on time or space, demonstrating the rich probabilistic behavior generated by inhomogeneity.

In this paper, following the lineage of these preceding studies, we propose a novel inhomogeneous random walk model with a spatial decay parameter $r \in (0, 1]$. This model combines the following two core features:
\begin{enumerate}[i.]
 \item At the start of each excursion, the parameter $p$ is independently resampled from a uniform distribution $U(0, 1)$ (temporal renewal structure).
 \item The probability of moving to the right from state $k$ decays exponentially as $r^k p$, creating a stronger pull-back toward the origin as the process moves farther away (spatial inhomogeneity).
\end{enumerate}

The structure of this paper is as follows. First, in Section 2, we clearly state the setup of the stochastic process. Then, in Section 3, we present the main results of this paper. We begin by deriving the hitting probability to a specific boundary $N$ using a scale function. Next, we provide the probability generating function of the hitting time and the expected value of the occupation time. Finally, we analyze the distribution and expected value of the maximum penetration depth during an excursion.

\section{Model Description}

Let $(X_t)_{t \ge 0}$ be a one-dimensional discrete-time stochastic process taking values in the state space $\mathcal{S} = \{0, 1, 2, \dots\}$. This process $(X_t)_{t \ge 0}$ represents an intrinsically inhomogeneous, space-dependent random walk modulated by a random environmental variable.

The trajectory of this walk is decomposed into a sequence of independent trials called excursions. Each excursion always starts from state $0$. At the beginning of the $m$-th excursion, an environmental parameter $Z_m$ is sampled independently from a continuous uniform distribution, i.e., $Z_m \sim U(0,1)$. This parameter is kept constant throughout the duration of that specific excursion.

Conditional on the sampled parameter $Z_m = p$, at the initial state $k=0$, the process starts an excursion by moving one step to the right with probability $p$, and staying in place with probability $1-p$:
\begin{align}
    P(X_{t+1} = 1 \mid X_t = 0, Z_m = p) &= p, \\
    P(X_{t+1} = 0 \mid X_t = 0, Z_m = p) &= 1 - p.
\end{align}
The transition probabilities of $X_t$ from state $k \in \{1, 2, \dots\}$ are defined as follows:
\begin{align}
    P(X_{t+1} = k+1 \mid X_t = k, Z_m = p) &= r^k p, \\
    P(X_{t+1} = k-1 \mid X_t = k, Z_m = p) &= 1 - r^k p,
\end{align}
where $r \in (0, 1]$ is a fixed spatial decay parameter. The factor $r^k$ introduces a spatial constraint. As the particle moves to the right (in the direction of increasing $k$), the probability of continuing to move right decays exponentially, generating a strong drift that pulls it back toward the origin.

When the process returns to $0$, the current excursion ends, a new environmental parameter $Z_{m+1} \sim U(0,1)$ is independently sampled, and the next excursion starts from state $0$. This resampling procedure introduces a renewal structure into the sequence of excursions.

For much of this paper, we deal with a killed process, where the process is terminated if it reaches an upper boundary $N>0$ before returning to the origin $0$. In that case, the duration of the excursion is given by the stopping time $\min(\theta, \tau_N)$, where $\theta = \inf\{t > 0 \mid X_t = 0\}$ is the first return time to $0$, and $\tau_N = \inf\{t > 0 \mid X_t = N\}$ is the first hitting time to $N$.

\section{Results}

For $i = 0, 1, 2, \dots$, let $\theta_i$ be the $i$-th time that $X$ reaches $0$. Specifically, we set $\theta_0 = 0$ and define $\theta_i = \inf \{ u > \theta_{i-1} \mid X_u = 0 \}$ for $i \geq 1$.

Furthermore, we denote the excursions as $e_1, e_2, e_3, \dots$, that is, $e_i=(X_{\theta_{i-1}}, X_{\theta_{i-1}+1}, \dots, X_{\theta_{i}})$ for $i=1, 2, \dots$.

\subsection{First Hitting Time}

First, we investigate the distribution of the height of the excursion.

\begin{theorem}
\label{thm1}
Let the scale function defined for $k \ge 1$ be
\begin{align}
    S(k;p) = \sum_{i=0}^{k-1} \prod_{j=1}^{i} \frac{1 - r^j p}{r^j p}
\end{align}
where the empty product for $i=0$ is defined as $1$. Then, the probability that the first excursion $e_1$ does not reach $N$ is given by:
\begin{align}
    P(e_1 \text{ does not reach } N) = 1 - \int_0^1 \frac{p}{S(N;p)} dp.
\end{align}
\end{theorem}

\begin{proof}
Suppose the condition $Z_1=p$ is given. Let $a(k;p)$ be the probability that the random walk $X$ starting from state $k \in \{1, \dots, N-1\}$ reaches $0$ before reaching $N$. By the Markov property, $a(k;p)$ satisfies the following linear difference equation:
\begin{align}
    a(k;p) = r^k p a(k+1;p) + (1 - r^k p) a(k-1;p).
\end{align}
The boundary conditions are $a(0;p) = 1$ and $a(N;p) = 0$. Rewriting this equation, we obtain:
\begin{align}
    r^k p (a(k+1;p) - a(k;p)) = (1 - r^k p) (a(k;p) - a(k-1;p)).
\end{align}
From the boundary condition $u(0;p)=1$, this difference equation yields
\begin{align}
      \frac{a(k+1;p) - a(k;p)}{a(k;p) - a(k-1;p)} &= \frac{1 - r^k p}{r^k p}
      \\
     \frac{a(i;p) - a(i-1;p)}{a(1;p) - a(0;p)}&= \prod_{j=0}^{i-1} \frac{(1 - r^j p)}{r^j p}
     \\
     a(i;p) - a(i-1;p)&=(a(1;p) - 1)\prod_{j=0}^{i-1} \frac{(1 - r^j p)}{r^j p}
\end{align}
Therefore, we have
\begin{align}
     a(k;p)-1=(a(1;p) - 1)\sum_{i=1}^{k} \prod_{j=0}^{i-1} \frac{(1 - r^j p)}{r^j p}
     =(a(1;p) - 1)S(k;p).
\end{align}
When $k=N$, $a(N;p)-1=(a(1;p) - 1)S(N;p)$ implies that, from $a(N;p)=0$, 
\begin{align}
   a(1;p) - 1=-\frac{1}{S(N;p)}
\end{align}
Therefore, we obtain
\begin{align}
  a(k;p)-1&=-\frac{S(k;p)}{S(N;p)}
  \\
  a(k;p)&=\frac{S(N;p)-S(k;p)}{S(N;p)}.
\end{align}
In particular, the probability of reaching $0$ before $N$ starting from $k=1$ is
\begin{align}
    a(1;p) = 1 - \frac{S(1;p)}{S(N;p)} = 1 - \frac{1}{S(N;p)}.
\end{align}
The event that $e_1$ does not reach $N$ occurs if the first step does not move with probability $1-Z_1$ (in which case it terminates immediately), or if the first step is to the right with probability $Z_1$ and subsequently reaches $0$ before hitting $N$. Taking the expectation with respect to the uniform random variable $Z_1 \in (0,1)$,
\begin{align}
    P(e_1 \text{ does not reach } N) &= E[ 1 - Z_1 + Z_1 a(1;Z_1) ] \nonumber \\
    &= \int_0^1 \left( 1 - p + p \left( 1 - \frac{1}{S(N;p)} \right) \right) dp \nonumber \\
    &= 1 - \int_0^1 \frac{p}{S(N;p)} dp
\end{align}
which completes the proof.
\end{proof}

We define the local time of $X$ as $L_t = \max\{i \mid \theta_i \leq t\}$ ($t = 0, 1, 2, \dots$). Then, $L_{\tau_N}$ represents the number of excursions before $X$ reaches $N$ for the first time. The probability distribution of $L_{\tau_N}$ is as follows.

\begin{corollary}
For $n=0, 1, 2, \dots$,
\begin{align}
    P(L_{\tau_N}=n) = \left( 1 - \int_0^1 \frac{p}{S(N;p)} dp \right)^n \int_0^1 \frac{p}{S(N;p)} dp ,
\end{align}
where $S(N;p)$ is the scale function defined in Theorem \ref{thm1}.
\end{corollary}

\begin{proof}
The event $\{L_{\tau_N} = n\}$ means that the first $n$ excursions return to $0$ without hitting $N$, and the $(n+1)$-th excursion successfully reaches $N$. Since the random parameter $Z$ is independently resampled at the start of each excursion, the sequence of excursions forms a sequence of independent trials. Therefore,
\begin{align}
    P(L_{\tau_N}=n) &= P(e_1 \text{ does not reach } N)^n \cdot P(e_1 \text{ reaches } N) \nonumber \\
    &= P(e_1 \text{ does not reach } N)^n \cdot \big( 1 - P(e_1 \text{ does not reach } N) \big).
\end{align}
By substituting the explicit probabilities derived in Theorem \ref{thm1}, the geometric distribution as claimed is directly obtained.
\end{proof}

We seek the length of the excursion when it is not killed. It can be characterized by the solution to a difference equation.
Hereafter, let $I(A)$ denote the indicator function of an event $A$, that is, $I(A)=1$ if $A$ occurs, and $I(A)=0$ if $A$ does not occur.

\begin{theorem}
\label{thm2}
For $0 < q < 1$,
\begin{align}
    E[q^{\theta_1} I(e_1 \text{ does not reach } N)] = \frac{q}{2} + \int_0^1 p q b(1;p) dp,
\end{align}
where $b(k;p)$ is the unique solution to the following linear boundary value problem:
\begin{align}
    b(k;p) = q r^k p b(k+1;p) + q (1 - r^k p) b(k-1;p), \quad k = 1, 2, \dots, N-1,
\end{align}
with boundary conditions $b(0;p) = 1$ and $b(N;p) = 0$.
\end{theorem}

\begin{proof}
Let $b(k;p) = E[q^{\tau_0} I(\tau_0 < \tau_N) \mid X_0 = k, Z_1 = p]$ be the probability generating function of the time to return to $0$ before reaching $N$, starting from state $k$. By conditioning on the first step, $b(k;p)$ naturally satisfies the following linear difference equation:
\begin{align}
    b(k;p) = q r^k p b(k+1;p) + q (1 - r^k p) b(k-1;p)
\end{align}
for $1 \leq k \leq N-1$. Since the process stops and the excursion duration is realized upon returning to $0$, we have $b(0;p) = 1$. On the other hand, if the process reaches $N$ before $0$, the excursion is killed, yielding $b(N;p) = 0$. 

At the beginning of the first excursion from state $0$, conditional on $Z_1 = p$, the process either stays at $0$ for one step with probability $1-p$ (in which case the excursion ends immediately at $t=1$), or moves to $1$ with probability $p$. Taking the expectation with respect to $Z_1 \sim U(0,1)$, we obtain:
\begin{align}
    E[q^{\theta_1} I(e_1 \text{ does not reach } N)] &= E[ (1 - Z_1) q + Z_1 q b(1;Z_1) ] \nonumber \\
    &= \int_0^1 (1 - p) q dp + \int_0^1 p q b(1;p) dp \nonumber \\
    &= \frac{q}{2} + \int_0^1 p q b(1;p) dp.
\end{align}
The sequence $b(k;p)$ can be obtained explicitly by solving the associated tridiagonal linear system, which completes the proof.
\end{proof}

For $0 < q < 1$, the probability generating function of the time to reach $N$ during the first excursion is given as follows:
\begin{theorem}
\label{thm3}

For $0 < q < 1$,
\begin{align}
    E[q^{\tau_N} I(e_1 \text{ reaches } N)] = \int_0^1 p q c(1;p) dp,
\end{align}
where $c(k;p)$ is the unique solution to the following linear boundary value problem:
\begin{align}
    c(k;p) = q r^k p c(k+1;p) + q (1 - r^k p) c(k-1;p), \quad k = 1, 2, \dots, N-1,
\end{align}
with boundary conditions $c(0;p) = 0$ and $c(N;p) = 1$.
\end{theorem}

\begin{proof}
Let $c(k;p) = E[q^{\tau_N} I(\tau_N < \tau_0) \mid X_0 = k, Z_1 = p]$ be the probability generating function of the time to reach $N$ before returning to $0$, starting from state $k$. By conditioning on the first step, $c(k;p)$ naturally satisfies the following linear difference equation:
\begin{align}
    c(k;p) = q r^k p c(k+1;p) + q (1 - r^k p) c(k-1;p)
\end{align}
for $1 \leq k \leq N-1$. Since the process stops upon reaching $0$ or $N$, the boundary conditions are exactly $c(0;p) = 0$ and $c(N;p) = 1$. The event that $e_1$ reaches $N$ is possible only if the first step is to the right. Taking the expectation with respect to $Z_1 \sim U(0,1)$, we obtain:
\begin{align}
    E[q^{\tau_N} I(e_1 \text{ reaches } N)] = E[ Z_1 q c(1;Z_1) ] = \int_0^1 p q c(1;p) dp.
\end{align}
The sequence $c(k;p)$ can be obtained explicitly by solving the associated tridiagonal linear system.
\end{proof}

Using the explicit characterization of successful and failed excursions from Theorem \ref{thm2} and \ref{thm3}, we can derive the probability generating function of the total time required for the first hitting of $N$.

\begin{corollary}
\label{thm4}
For $0 < q < 1$, the probability generating function of the first hitting time $\tau_N$ is given by:
\begin{align}
    E[q^{\tau_N}] = \frac{ \displaystyle \int_0^1 p q c(1;p) dp }{ \displaystyle 1 - \left( \frac{q}{2} + \int_0^1 p q b(1;p) dp \right) },
\end{align}
where $b(k;p)$ and $c(k;p)$ are the solutions to the linear difference equations defined in Theorem \ref{thm2} and Theorem \ref{thm3}, respectively.
\end{corollary}

\begin{proof}
The total time $\tau_N$ to reach $N$ can be decomposed into the sum of the durations of the first $L_{\tau_N}$ failed excursions that returned to the origin, and the duration of the final successful excursion that reached $N$. Since the parameter $Z$ is independently resampled from the uniform distribution $U(0,1)$ at the start of each excursion, the durations of these excursions are mutually independent. Then,
\begin{align}
    E[q^{\tau_N}] &= \sum_{n=0}^{\infty} E\left[ q^{\tau_N} I( L_{\tau_N} = n) \right] \\
    &=\sum_{n=0}^{\infty} E\left[\prod_{i=1}^n q^{\theta_i-\theta_{i-1}}I(e_i \text{ does not reach } N)q^{\tau_N-\theta_{n}}I(e_{n+1} \text{ reachs } N)  \right]
    \\
    &= \sum_{n=0}^{\infty} \left( E[q^{\theta_1} I(e_1 \text{ does not reach } N)] \right)^n E[q^{\tau_N} I(e_1 \text{ reaches } N)].
\end{align}
This is a geometric series. Since $E[q^{\theta_1} I(e_1 \text{ does not reach } N)] \le P(e_1 \text{ does not reach } N) < 1$ for $q \in (0,1)$, the series converges to:
\begin{align}
    E[q^{\tau_N}] = \frac{ E[q^{\tau_N} I(e_1 \text{ reaches } N)] }{ 1 - E[q^{\theta_1} I(e_1 \text{ does not reach } N)] }.
\end{align}
Substituting the explicit formulas obtained in Theorem \ref{thm2} and \ref{thm3} into the numerator and denominator completes the proof.
\end{proof}

By evaluating the derivative of the probability generating function as $q \nearrow 1$, we can obtain the expected time to reach $N$.

\begin{corollary}
\label{cor:expected_time}
Let $D_1$ be the duration of the first excursion regardless of whether it terminates at $0$ or $N$, i.e., $D_1=\min\{\theta_1,\tau_N\}$. The expected value of the first hitting time to $N$ is given by:
\begin{align}
    E[\tau_N] = \frac{E[D_1]}{ \displaystyle \int_0^1 \frac{p}{S(N;p)} dp },
\end{align}
where $E[D_1] = \frac{d}{dq} \Big( E[q^{\theta_1} I(e_1 \text{ does not reach } N)] + E[q^{\tau_N} I(e_1 \text{ reaches } N)] \Big) \Big|_{q=1}$, and $S(N;p)$ is the scale function defined in Theorem \ref{thm1}.
\end{corollary}

\begin{proof}
Let $A(q) = E[q^{\theta_1} I(e_1 \text{ does not reach } N)]$ and $B(q) = E[q^{\tau_N} I(e_1 \text{ reaches } N)]$. From Theorem \ref{thm4}, we have $E[q^{\tau_N}] = \frac{B(q)}{1 - A(q)}$. Differentiating with respect to $q$ yields:
\begin{align}
    \frac{d}{dq} E[q^{\tau_N}] = \frac{B'(q)(1 - A(q)) + B(q)A'(q)}{(1 - A(q))^2}.
\end{align}
Taking the limit as $q \nearrow 1$, we have $A(1) = P(e_1 \text{ does not reach } N)$ and $B(1) = P(e_1 \text{ reaches } N)$. Therefore, by Theorem \ref{thm1}, $1 - A(1) = B(1) = \int_0^1 \frac{p}{S(N;p)} dp$. Substituting these into the derivative expression gives:
\begin{align}
    E[\tau_N] &= \frac{B'(1)B(1) + B(1)A'(1)}{(B(1))^2} \nonumber \\
    &= \frac{A'(1) + B'(1)}{B(1)}.
\end{align}
Since $A'(1) + B'(1) = E[\theta_1 I(e_1 \text{ does not reach } N)] + E[\tau_N I(e_1 \text{ reaches } N)] = E[D_1]$, the desired result is obtained.
\end{proof}

\subsection{Occupation Time}

Given $Z_1 = p$, we investigate the expected number of visits to state $k \in \{1, 2, \dots, N-1\}$ during the first excursion.

\begin{theorem}\label{thm:greens_function}
For $k \in \{1, 2, \dots, N-1\}$, let
$$G(k;p) = E\left[\left. \sum_{t=0}^{\tau_0 \wedge \tau_N - 1} I(X_t = k) \right| X_0 = 1, Z_1 = p \right].$$
Then $G(k;p)$ is given by:
$$    G(k;p) = \frac{S(N;p) - S(k;p)}{W(k;p) S(N;p)},
$$
where $S(k;p)$ is the scale function defined in Theorem \ref{thm1}, and $W(k;p)$ is defined as follows:
$$    W(k;p) = r^k p \prod_{j=1}^k \frac{1 - r^j p}{r^j p}.$$
\end{theorem}

\begin{proof}
Let $G(x, k; p)$ be the expected number of visits to state $k$ starting from state $x \in \{0, 1, \dots, N\}$ before reaching state $0$ or $N$. By definition, $G(k;p) = G(1, k; p)$. Using the Markov property, $G(x, k; p)$ satisfies the following linear boundary value problem:
$$    G(x, k; p) = r^x p G(x+1, k; p) + (1 - r^x p) G(x-1, k; p) + \delta_{x,k},$$
where the boundary conditions are $G(0, k; p) = 0$ and $G(N, k; p) = 0$, and $\delta_{x,k}$ is the Kronecker delta. For $x <k$, this equation is identical to the homogeneous difference equation solved in Theorem \ref{thm1}, and similarly satisfies
$$G(x+1,k;p)-G(x,k;p)=\frac{1-r^xp}{r^xp}(G(x,k;p)-G(x-1,k;p)).$$
Therefore, for $x < k$, the solution is
\begin{align}
  G(x,k;p)-G(0,k;p)&=G(1,k;p) \sum_{i=1}^{x} \prod_{j=1}^{i-1} \frac{1-r^jp}{r^jp}
  \\
  G(x,k;p)&=G(1,k;p)S(x;p)
\end{align}
and thus is proportional to $S(x;p)$.
 Also, for $x > k$,
 \begin{align}
   G(i,k;p)-G(i-1,k;p)&=(G(k+1,k;p)-G(k,k;p))\prod_{j=k+1}^{i-1} \frac{1-r^jp}{r^jp}
   \\
   G(N,k;p)-G(x,k;p)&=(G(k+1,k;p)-G(k,k;p))\sum_{i=x+1}^{N}\prod_{j=k+1}^{i-1} \frac{1-r^jp}{r^jp}
   \\
   G(x,k;p)&=-(G(k+1,k;p)-G(k,k;p))(S(N;p)-S(x;p))
 \end{align}
 it must be proportional to $S(N;p) - S(x;p)$.
To ensure continuity at $x=k$, the solution can be constructed as follows:
$$    G(x, k; p) = 
    \begin{cases}
        G(k, k; p) \dfrac{S(x;p)}{S(k;p)} & (0 \le x \le k) \\[8pt]
        G(k, k; p) \dfrac{S(N;p) - S(x;p)}{S(N;p) - S(k;p)} & (k \le x \le N)
    \end{cases}$$
To find $G(k, k; p)$, we apply the difference equation at $x=k$:
$$    G(k, k; p) = r^k p G(k+1, k; p) + (1 - r^k p) G(k-1, k; p) + 1.$$
Substituting the proportional forms of $G(k+1, k; p)$ and $G(k-1, k; p)$ gives:
\begin{gather}
  G(k, k; p) \left[ 1 - r^k p \frac{S(N;p) - S(k+1;p)}{S(N;p) - S(k;p)} - (1 - r^k p) \frac{S(k-1;p)}{S(k;p)} \right] = 1
  \\
  G(k, k; p) \left[  r^k p \frac{S(k+1;p) - S(k;p)}{S(N;p) - S(k;p)} + (1 - r^k p) \frac{S(k;p)-S(k-1;p)}{S(k;p)} \right] = 1
\end{gather}
Using the property of the scale function $S(m+1;p) - S(m;p) = \prod_{j=1}^m \frac{1 - r^j p}{r^j p}$, we find that both $r^k p (S(k+1;p) - S(k;p))$ and $(1 - r^k p) (S(k;p) - S(k-1;p))$ are exactly equal to $W(k;p)$. Rearranging the terms inside the brackets leads to:
$$    G(k, k; p) \left[ W(k;p) \left( \frac{1}{S(N;p) - S(k;p)} + \frac{1}{S(k;p)} \right) \right] = 1,$$
which yields $G(k, k; p) = \dfrac{S(k;p) (S(N;p) - S(k;p))}{W(k;p) S(N;p)}$. Finally, evaluating $G(x, k; p)$ at the starting state $x=1$, we obtain:
$$    G(1, k; p) = G(k, k; p) \frac{S(1;p)}{S(k;p)} = \frac{S(N;p) - S(k;p)}{W(k;p) S(N;p)}$$
since $S(1;p) = 1$. This completes the proof.
\end{proof}

Using the expected number of visits during an excursion, the total expected duration  of the first excursion $E[D_1]$ where $D_1=\min\{\theta_1,\tau_N\}$ can be expressed.

\begin{corollary}
The expected duration $E[D_1]$ of the first excursion satisfies the following relationship:
$$    E[D_1] = 1 + \int_0^1 p \sum_{k=1}^{N-1} G(k;p) dp,
$$
where $G(k;p)$ is the Green's function derived in Theorem \ref{thm:greens_function}.
\end{corollary}

\begin{proof}
Conditional on $Z_1=p$, if the first step of the process does not move (probability $1-p$), it immediately reaches $0$, and the duration is $1$. If the first step is to the right (probability $p$), the expected time spent in the interior states before reaching $0$ or $N$ is $\sum_{k=1}^{N-1} G(k;p)$. Adding the first step, the conditional expected duration is
$$    E[D_1 \mid Z_1 = p] = (1-p) \cdot 1 + p \left( 1 + \sum_{k=1}^{N-1} G(k;p) \right) = 1 + p \sum_{k=1}^{N-1} G(k;p).$$
Since $Z_1$ follows a uniform distribution, integrating over $p \in (0,1)$ yields the unconditional expected value $E[D_1]$.
\end{proof}

\subsection{Maximum Penetration Depth}

To further understand the behavior of a single excursion, we investigate its maximum penetration depth. Let $\displaystyle M = \max_{0 \le t \le \theta_1} X_t$ be the maximum state reached during the first excursion $e_1$ before returning to the origin.

\begin{theorem}
\label{thm:max_penetration}
The tail distribution of the maximum penetration depth $\displaystyle M = \max_{0 \le t \le \theta_1} X_t$ is given by:
\begin{align}
    P(M \ge k) = \int_0^1 \frac{p}{S(k;p)} dp, \quad (k = 1, 2, \dots)
\end{align}
where $S(k;p)$ is the scale function from Theorem \ref{thm1}.
\end{theorem}

\begin{proof}
The event $\{M \ge k\}$ means that the random walk successfully reaches state $k$ during the first excursion before returning to $0$. This is equivalent to the event that $e_1$ hits the boundary when the boundary set is placed at $k$. From the non-hitting probability derived in Theorem \ref{thm1} (with $N$ replaced by $k$), we find:
\begin{align}
    P(M \ge k) = 1 - P(e_1 \text{ does not reach } k) = \int_0^1 \frac{p}{S(k;p)} dp,
\end{align}
so that we obtain the assertion. 
\end{proof}

Since $M$ is a non-negative integer-valued random variable, its expected value can be calculated using the tail sum formula:
\begin{align}
    E[M] = \sum_{k=1}^{\infty} P(M \ge k) = \sum_{k=1}^{\infty} \int_0^1 \frac{p}{S(k;p)} dp.
\end{align}
Since the integrand is non-negative, the order of summation and integration can be interchanged by the monotone convergence theorem (or Fubini-Tonelli theorem).

 \bibliography{5ref} 
 \bibliographystyle{apalike}

\end{document}